\documentclass[a4paper]{article}
\usepackage[utf8]{inputenc}
\usepackage[english]{babel}
\usepackage{amsmath}
\usepackage{amssymb}
\usepackage{amsthm}
\usepackage{hyperref}
\usepackage{enumerate}
\usepackage{a4wide}
\usepackage{pgfplotstable}
\usepackage{pgfplots}
\usepackage[ruled,linesnumbered]{algorithm2e}

\newtheorem{theorem}{Theorem}

\newtheorem{proposition}[theorem]{Proposition}
\newtheorem{corollary}[theorem]{Corollary}
\theoremstyle{definition}
\newtheorem{example}[theorem]{Example}
\newtheorem{remark}[theorem]{Remark}

\newcommand{\R}{\mathbb{R}}
\newcommand{\N}{\mathbb{N}}

\DeclareMathOperator{\lineal}{lineal}
\DeclareMathOperator{\cone}{cone}

\DeclareMathOperator{\dom}{dom}
\DeclareMathOperator{\epi}{epi}

\DeclareMathOperator{\cl}{cl}

\newcommand*{\colonequals}{\mathrel{\vcenter{\baselineskip0.5ex%
\lineskiplimit0pt\hbox{\scriptsize.}\hbox{\scriptsize.}}}=}

\newcommand\st{\,\middle\vert\;} 

\begin{document}
\title{Solving polyhedral d.c.\ optimization problems via concave minimization}

\author{Simeon vom Dahl\textsuperscript{1}
  \and
  Andreas Löhne\textsuperscript{2}
}

\footnotetext[1]{Université Paris-Sud, Department of Mathematics, 91400 Orsay, France}
\footnotetext[2]{Friedrich Schiller University Jena, Department of
Mathematics, 07737 Jena, Germany, andreas.loehne@uni-jena.de}

\maketitle

\begin{abstract} The problem of minimizing the difference of two convex functions is called polyhedral d.c.\ optimization problem if at least one of the two component functions is polyhedral. We characterize the existence of global optimal solutions of polyhedral d.c.\ optimization problems. This result is used to show that, whenever the existence of an optimal solution can be certified, polyhedral d.c.\ optimization problems can be solved by certain concave minimization algorithms. No further assumptions are necessary in case of the first component being polyhedral and just some mild assumptions to the first component are required for the case where the second component is polyhedral. In case of both component functions being polyhedral, we obtain a primal and dual existence test and a primal and dual solution procedure. Numerical examples are discussed.
	
\medskip

\end{abstract}

\section{Introduction}

D.c.\ programming and concave minimization are known to be closely related problems, see e.g. \cite{Tuy87}. Theory and methods for concave minimization are surveyed, for instance, in \cite{Benson95}. An overview about the field of d.c.\ programming is given, for example, in \cite{HT99,Tuy95}. 

Let $g \colon \R^n \to \R\cup\{\infty\}$ and $h \colon \R^n \to \R\cup\{\infty\}$ be convex functions, where one of the functions $g$ or $h$ is assumed to be {\em polyhedral}, i.e., the epigraph of the respective function is a convex polyhedron. We consider the {\em polyhedral d.c.\ optimization problem}
\begin{equation}
\min_{x\in \dom g} [g(x)-h(x)].
\label{P}
\tag{DC}
\end{equation}
This problem will be transformed into a concave minimization problem under linear constraints. The general form of such a problem is as follows. For a concave function $f:\R^k \to \R\cup \{-\infty\}$ we consider
\begin{equation} \label{ConcMinGen}
	\min_y f(y) \quad\text{ s.t. }\quad y \in P,
\end{equation}
where the feasible set $P$ is an arbitrary polyhedron. For the reformulation of \eqref{P} we choose the concave objective function
\begin{equation}\label{obj}
	f(x,r) \colonequals r-h(x).
\end{equation}
The feasible set $P$ is the epigraph of $g$. The concave minimization problem associated to \eqref{P}
\begin{equation}\label{ConcMinDC}\tag{ConcMin}
	\min_{x,r} f(x,r) \quad \text{ s.t. } \quad (x,r) \in \epi g
\end{equation}
is equivalent to \eqref{P} in the following sense:
\begin{itemize}
	\item If $(x^0,r^0)$ solves \eqref{ConcMinDC} then $x^0$ solves \eqref{P}.
	\item If $x^0$ solves \eqref{P} then $(x^0,g(x^0))$ solves \eqref{ConcMinDC}.
\end{itemize}

In the literature on concave minimization, many authors assume a compact feasible set in order to guaranty the existence of optimal solutions, see e.g. \cite{comp_feas_set_3,comp_feas_set_1,comp_feas_set_2}. However, problem \eqref{ConcMinDC} always has a non-compact feasible set. In \cite{CirLoeWei18}, algorithms for concave (even quasi-concave) minimization based on a modification of methods for vector linear programming (VLP) are studied. An implementation of these methods based on the VLP solver {\em bensolve} \cite{bensolve_paper,bensolve} is provided by the Octave/Matlab package {\em bensolve tools} \cite{bt_paper,bt}. This approach allows non-compact feasible sets but requires certain other assumptions.

While the reformulation of \eqref{P} as \eqref{ConcMinDC} is straightforward, our research focuses on the assumptions which are required for the solution methods. It turns out that verifying or disproving the existence of optimal solutions of \eqref{P} is the crucial task here. For the case where $g$ is polyhedral, we prove that whenever an optimal solution of \eqref{P} exists, it can be computed by solving the associated problem \eqref{ConcMinDC} using the methods of \cite{CirLoeWei18}. In the case where $h$ is polyhedral, the same applies to the dual problem of \eqref{P}. Under mild assumptions, an optimal solution of \eqref{P} can be obtained from an optimal solution of the dual problem.

This article is organized as follows. In Section 2 we characterize the existence of optimal solutions for polyhedral d.c.\ programs.  Section 3 explains how a polyhedral d.c.\ program can be solved using a (quasi-)concave minimization solver like {\em bensolve tools} after the existence of optimal solutions has been verified. The last section presents two numerical examples for the case where both $g$ and $h$ are polyhedral. Both  existence tests and solution methods are addressed.

We use the following notation.
The {\em domain} of a function $f: \R^n \to \R\cup\{\infty\}$ is defined by $\dom f \colonequals \{x \in \R^n \mid f(x) < \infty\}$ and the epigraph of $f$ is the set $\epi f \colonequals \{(x,r) \in \R^n\times \R \mid r \geq f(x)\}$.
A convex function $f$ is called {\em closed} if $\epi f$ is a closed set.
We write $\R^n_+$ for the set of vectors with non-negative components.
The {\em recession cone} $0^+C$ of a convex set $C \subseteq \R^n$ is the set of all $y$ with $C + \{y\} \subseteq C$.
The {\em lineality space} of $C$ is the set $\lineal(C)\colonequals 0^+C\cap(-0^+C)$.

\section{Existence of global optimal solutions}

In this section we discuss the existence of optimal solutions of problem \eqref{P} for the following three cases: (1) $g$ being polyhedral, (2) $h$ being polyhedral, (3) both $g$ and $h$ being polyhedral.

\subsection{The case of $g$ being polyhedral}

The following characterization of the existence of optimal solutions is the main result of this article.

\begin{theorem}\label{th1}
Problem \eqref{P} with $g$ being polyhedral has an optimal solution if and only if the following three properties hold:
\begin{enumerate}[(i)]
\item $\dom g \neq \emptyset$,
\item $\dom g \subseteq \dom h$,
\item $0^+\epi g \subseteq 0^{+}(\epi h \cap (\dom g \times \mathbb{R}))$.
\end{enumerate}
\end{theorem}

\begin{proof}
Let \eqref{P} have an optimal solution $x^0$. Since $x^0 \in \dom g$, (i) is satisfied. Let $x \in \dom g$, then 
$-h(x) \geq g(x^0)-h(x^0)-g(x) > -\infty$ and hence $x \in \dom h$, i.e., (ii) holds. Assume (iii) is violated, that is,
$0^+\epi g \nsubseteq 0^+(\epi h \cap (\dom g \times \R))$. Then we can choose $(d,s) \in 0^+ \epi g\setminus 0^{+}(\epi h \cap (\dom g \times \mathbb{R}))$. By the definition of the recession cone, there exists some $(x,t) \in \epi h\cap (\dom g \times \mathbb R)$ such that $(x,t)+\alpha (d,s) \notin \epi h \cap (\dom g \times \R)$ for some $\alpha > 0$. Without loss of generality we can set $\alpha = 1$. Since $t \geq h(x)$ we get
$(x,h(x))+ (d,s) \notin \epi h \cap (\dom g \times \mathbb{R})$. We find some $r \in \R$ such that $(x,r) \in \epi g$.

Case 1: If $(x,h(x))+(d,s) \notin \dom g \times \mathbb{R}$, then $(x,r)+ (d,s) \notin \dom g \times \mathbb{R}$ and hence $(d,s) \not\in 0^+\epi g$, a contradiction.

Case 2: If $(x,h(x))+ (d,s) \notin \epi h$, then 
\begin{equation}\label{eq1}
	h(x)+s < h(x+d)
\end{equation}
by definition of the epigraph. 
We consider problem \eqref{ConcMinDC} which is equivalent to \eqref{P} as discussed above. For its objective function $f$ defined in \eqref{obj} we obtain
\begin{equation}\label{eq2}
	f(x+d,r+s) < f(x,r).
\end{equation}
For any $n \in \N$, $(x+nd,r+ns)$ is feasible for \eqref{ConcMinDC}. But $f(x+nd,r+ns)$ tends to $-\infty$, by concavity of $f$ and \eqref{eq2}. This contradicts the assumption that \eqref{P} has an optimal solution. Thus, (iii) is satisfied.

Assume now that (i), (ii) and (iii) hold. By (i), $\epi g$ is non-empty. Let $(x,r) \in \epi g$ and $(d,s) \in 0^+ \epi g$. Then $(x,r)+\alpha (d,s) \in \epi g$ for all $\alpha > 0$. By (ii), we have $(x,h(x)) \in \epi h \cap (\dom g \times \mathbb{R})$. By (iii) we obtain $(d,s) \in 0^{+}(\epi h \cap (\dom g \times \mathbb{R}))$. Thus, $(x,h(x))+\alpha (d,s) \in \epi h \cap (\dom g \times \mathbb{R})$ for all $\alpha > 0$. From the definition of the epigraph, we obtain $h(x) + \alpha s \geq h(x+\alpha d)$. For the objective function $f$ of problem \eqref{ConcMinDC}, this implies
\begin{equation}\label{eq3}
f(x+\alpha d, r+\alpha s) = r+\alpha s -h(x+\alpha d) \geq r-h(x) = f(x,r).
\end{equation}
Since $\epi g$ is a polyhedron, it can be expressed by a polytope $Q$ as $\epi g = Q + 0^+ \epi g$. By (ii), the values of $f$ are finite over $Q$. The polytope $Q$ is the convex hull of its vertices. Concavity of $f$ implies that $f$ attains its minimum at a vertex $(x^0,r^0)$ of $Q$. 
Since \eqref{eq3} holds for every $(x,r) \in Q$ and every $(d,s) \in 0^+ \epi g$, $(x^0,r^0)$ is an optimal solution of \eqref{ConcMinDC} and hence $x^0$ is an optimal solution of \eqref{P}.
\end{proof}

If $g$ is not polyhedral, the conditions (i), (ii) and (iii) in Theorem \ref{th1} are still necessary for the existence of optimal solutions. This can be seen in the first part of the proof, where the assumption of $g$ being polyhedral was not used. However, the conditions are no longer sufficient, not even if $h$ is polyhedral, as the following simple example shows.

\begin{example} Let $g,h:\R \to \R\cup\{\infty\}$ be defined as
	\begin{equation*}
		g(x) \colonequals \left\{\begin{array}{cl}
			- \sqrt{x} & \text{ if } x \geq 0\\
			+\infty & \text{ otherwise } 
		\end{array}\right. \qquad \text{and}\qquad
		h(x) \colonequals \left\{\begin{array}{cl}
			0 & \text{ if } x \geq 0\\
			+\infty & \text{ otherwise } 
		\end{array}\right. .
	\end{equation*}
	Then \eqref{P} is unbounded and thus has no optimal solution. We have $\dom g = \dom h = \R_+$ and $0^+ \epi g = 0^+ \epi h = \R^2_+$. Thus, the conditions (i), (ii) and (iii) of Theorem \ref{th1} are satisfied.
\end{example}

An extension to the non-polyhedral case requires further assumptions as discussed in the following remark.

\begin{remark} The second part of the proof of Theorem \ref{th1} still works for non-polyhedral functions $g$ if $\epi g$ is of the special form $Q + 0^+ \epi g$ for some compact set $Q$ and if $h$ is assumed to be upper semicontinuous. Then, by the Weierstrass theorem, the minimum of the objective function $f$ of \eqref{ConcMinDC} is attained in $Q$.
\end{remark}

Under certain assumptions, condition (iii) in Theorem \ref{th1} can be simplified. We start with two propositions and formulate this result as a corollary of Theorem \ref{th1}.

\begin{proposition}\label{recc}
Let $A,B \subseteq \R^n$ be non-empty convex sets with $A \subseteq B$ and let $B$ be closed. Then $0^+A \subseteq 0^+B$.
\end{proposition}
\begin{proof}
This follows from \cite[Theorem 8.3]{Rockafellar72}, which states that for a non-empty closed convex set $B$, $d \in 0^+B$ if and only if there is some $x \in B$ satisfying $x+\alpha d \in B$ for all $\alpha \geq 0$. Let $d \in 0^+A$. By definition of the recession cone, we have $x+\alpha d \in A$ for all $x \in A$. Since $A \subseteq B$, we get $d \in 0^+B$.
\end{proof}

\begin{proposition}\label{p2}
Let $0^+\epi h=0^+\cl \epi h$. Then condition (iii) in Theorem \ref{th1} is equivalent to 
\begin{enumerate}[(iii')]
	\item $0^+ \epi g \subseteq 0^+ \epi h$.
\end{enumerate}

\end{proposition}
\begin{proof}
Let (iii) be satisfied. Then
\begin{equation*}
0^+\epi g \overset{\text{(iii)}}{\subseteq} 0^{+}(\underbrace{\epi h \cap (\dom g \times \mathbb{R})}_{\subseteq \cl \epi h}) \overset{\text{Prop.\text{ }\ref{recc}}}{\subseteq} 0^+\cl\epi h = 0^+\epi h.
\end{equation*}

Let (iii') be satisfied. Let $(d,s) \in 0^+\epi g$ and  $(x,r) \in \epi h \cap (\dom g \times \mathbb R)$. By (iii'), we have $(x,r)+\alpha (d,s) \in \epi h$ for all $\alpha \geq 0$. Assuming that $(x,r)+\alpha (d,s) \notin \dom g \times \mathbb R$, we get $(x,g(x))+\alpha (d,s) \notin \dom g \times \mathbb R$, which contradicts the precondition $(d,s) \in 0^+\epi g$. Consequently, $(x,r)+\alpha (d,s) \in \epi h \cap (\dom g \times \mathbb R)$ for all $\alpha \geq 0$.
\end{proof}

\begin{corollary}\label{cor1}
Problem \eqref{P} with $g$ being polyhedral and $h$ being closed has an optimal solution if and only if the following properties hold:
\begin{enumerate}[ (i)]
\item $\dom g \neq \emptyset$,
\item $\dom g \subseteq \dom h$,
\item [(iii')] $0^+\epi g \subseteq 0^+ \epi h$.
\end{enumerate}	
\end{corollary}
\begin{proof} If $h$ is closed, we have $\epi h = \cl\epi h$ and hence $0^+\epi h = 0^+\cl\epi h$. Thus, the result follows from Proposition \ref{p2} and Theorem \ref{th1}.	
\end{proof}

The following example shows that condition (iii') is not adequate if $h$ is not assumed to be closed.

\begin{example}
Consider problem \eqref{P} for the functions $g, h:\R^2 \to \R \cup \{\infty\}$ with
\begin{equation*}
g(x,y)=\begin{cases} |x| & \text{if }y \in [1,\infty) \\ \infty & \text{otherwise}\end{cases}
\quad \text{ and } \quad
h(x,y)=\begin{cases} |x| & \text{if }y \in (0,\infty) \\ 2|x| & \text{if }y=0 \\ \infty & \text{otherwise}\end{cases}.
\end{equation*}
Both $g$ and $h$ are convex and $g$ is polyhedral.
Both functions coincide on $\dom g=\R \times [1,\infty)$, whence \eqref{P} has optimal solutions of the form $(0,y)$ for $y \geq 1$. The recession cones of the functions are
\begin{equation*} 
0^+\epi g=\cone\left\{\left(\begin{array}{r} -1 \\ 0 \\ 1 \end{array}\right),\left(\begin{array}{c} 1 \\ 0 \\ 1 \end{array}\right),\left(\begin{array}{c} 0 \\ 1 \\ 0 \end{array}\right)\right\}
\end{equation*}
and
\begin{equation*}
0^+\epi h=\cone\left\{\left(\begin{array}{r} -1 \\ 0 \\ 2 \end{array}\right),\left(\begin{array}{c} 1 \\ 0 \\ 2 \end{array}\right),\left(\begin{array}{c} 0 \\ 1 \\ 0 \end{array}\right)\right\}.
\end{equation*}
We see that $(1,0,1) \in 0^+\epi g \setminus 0^+\epi h$, i.e., (iii') is violated.
\end{example}

\subsection{The case of $h$ being polyhedral}

We consider the Toland-Singer dual problem  of \eqref{P}, see \cite{Sin79, Tol78}, that is,
\begin{equation}\label{D}\tag{DC*}
	\min_{y \in \dom h^*} [h^*(y) - g^*(y)],
\end{equation}
where $g^*(y) \colonequals \sup_{x \in \R^n} [{y}^Tx - g(x)]$ is the conjugate of $g$ and likewise for $h$.
The duality theory by Toland and Singer states that the optimal objective values of \eqref{P} and \eqref{D}
 coincide under the assumption of $h$ being closed. 

Since $g^*$ is convex and $h^*$ is polyhedral convex, the existence result of Theorem \ref{th1} applies to problem \eqref{D}. The following result provides the relation between optimal solutions of \eqref{P} and \eqref{D}. We denote by $\partial f(x)\colonequals \{y \in \R^n \mid \forall z \in \R^n: f(z) \geq f(x) + {y}^T(z-x)\}$ the {\em subdifferential} of a convex function $f:\R^n \to \R \cup \{ \infty\}$ at $x \in \dom f$. We set $\partial f(x) \colonequals \emptyset$ for $x \not\in\dom f$. 

\begin{proposition}{(e.g. \cite[Proposition 4.7]{HT99} or \cite[Proposition 3.20]{Tuy98})} \label{p_DC}
	Let $g:\R^n \to \R\cup\{\infty\}$ and $h:\R^n \to \R\cup\{\infty\}$ be convex functions with non-empty domain. Then:
	\begin{enumerate}[(i)]
	\item If $x$ is an optimal solution of \eqref{P}, then each $y \in \partial h(x)$ is an optimal solution of \eqref{D}.
	\end{enumerate}	
	If, in addition, $g$ and $h$ are closed, a dual statement holds:
	\begin{enumerate}[(i)]
	\item[(ii)] If $y$ is an optimal solution of \eqref{D}, then each $x \in \partial g^*(y)$ is an optimal solution of \eqref{P}.
	\end{enumerate}
\end{proposition}

\begin{remark} \label{rem1} 
As already mentioned in \cite{LoeWag17}, the assumption of $g$ and $h$ being closed for statement (ii) of Proposition \ref{p_DC} is missing in \cite[Proposition 4.7]{HT99}. In \cite[Proposition 3.20]{Tuy98} the assumption of $g$ being closed is required. Examples can be found in \cite{LoeWag17}.
\end{remark}

\begin{proposition}{(e.g. \cite[Theorem 23.5]{Rockafellar72})}\label{p5} Let $g:\R^n \to \R$ be a proper closed convex function. Then $x \in \partial g^*(y)$ if and only if $x$ is an optimal solution of 
	\begin{equation}
	\min\limits_{z \in\mathbb{R}^n}[g(z)-y^Tz].
	\label{p3}
	\end{equation}
\end{proposition}

\begin{theorem}\label{th2} Let $g$ be closed and let \eqref{p3} have an optimal solution for every $y \in \dom g^*$. Let $h$ be polyhedral.
	Then, problem
 \eqref{P} has an optimal solution if and only if the following properties hold:
	\begin{enumerate}[(i*)]
	\item $\dom h^* \neq \emptyset$,
	\item $\dom h^* \subseteq \dom g^*$,
	\item $0^+\epi h^* \subseteq 0^+ \epi g^*$.
	\end{enumerate}
\end{theorem}
\begin{proof}
	Let \eqref{P} have an optimal solution $x$. Since $x \in \dom h$ and $h$ is polyhedral, there exists some $y \in \partial h(x)$, see e.g. \cite[Theorem 23.10]{Rockafellar72}. Proposition \ref{p_DC} states that $y$ is an optimal solution of \eqref{D}. Theorem \ref{th1} applied to \eqref{D} yields the conditions (i*), (ii*) and (iii*).

Let the conditions (i*), (ii*) and (iii*)
 be satisfied. By Theorem \ref{th1} we obtain that \eqref{D} has an optimal solution $y$. By assumption, \eqref{p3} has an optimal solution $x$, which belongs to $\partial g^*(y)$, by Proposition \ref{p5}. Since $g$ and $h$ are closed, Proposition \ref{p_DC} yields that $x$ is an optimal solution to \eqref{P}.
\end{proof}

\subsection{The case of both $g$ and $h$ being polyhedral}

Combining the previous results we obtain the following statement.

\begin{corollary}\label{cor48} Let $g \colon \R^n \to \R\cup\{\infty\}$ and $h \colon \R^n \to \R\cup\{\infty\}$ are polyhedral convex functions. Then, the following statements are equivalent:
	\begin{enumerate}[(a)]
		\item Problem \eqref{P} has an optimal solution.
		\item Problem \eqref{D} has an optimal solution.
		\item The conditions (i), (ii), (iii') in Corollary \ref{cor1} are satisfied.
		\item The conditions (i*), (ii*), (iii*) in Theorem \ref{th2} are satisfied.
	\end{enumerate}
\end{corollary}
\begin{proof} By Corollary \ref{cor1}, (a) is equivalent to (c). Since $g$ and $h$ are polyhedral, the assumptions of Theorem \ref{th2} are satisfied. Indeed, by $y \in \dom g^*$, problem \eqref{p3} has a finite optimal value and hence the minimum is attained as $g$ is polyhedral. Theorem \ref{th2} yields that (b) is equivalent to (d). Proposition \ref{p_DC} and the fact that the subdifferential of a polyhedral function is non-empty at points of the domain of the function, we see that (a) is equivalent to (b).
\end{proof}

\section{Solution procedure}

Let $g$ be polyhedral in problem \eqref{P}. We are going to solve \eqref{P} by the following procedure. First we check whether or not an optimal solution of \eqref{P} exists by using Theorem \ref{th1}. If so, we solve the associated problem \eqref{ConcMinDC} by the solution methods of \cite{CirLoeWei18}. By using Theorem \ref{th1} again, we will check the following assumptions which are required to be satisfied for the algorithms presented in \cite{CirLoeWei18}:

There exists a polyhedral convex pointed cone $C \subseteq \R^n$ such that
\begin{itemize}
	\item [(M)] $f$ is $C$-monotone (i.e.\ $y-x \in C$ implies $f(x)\leq f(y)$),
	\item [(B)] $P$ is $C$-bounded (i.e.\ $0^+ P \subseteq C$).
\end{itemize}
If (M) and (B) are satisfied for some polyhedral convex pointed cone $C \subseteq \R^n$, then \eqref{ConcMinDC} has an optimal solution (\cite[Corollary 6]{CirLoeWei18}). Moreover, under these assumptions the methods in \cite{CirLoeWei18} compute optimal solutions of \eqref{ConcMinDC}, see \cite[Algorithm 2, Theorem 16]{CirLoeWei18} for the primal algorithm, \cite[Algorithm 4, Theorem 22]{CirLoeWei18} for the dual algorithm, and \cite[Section 6]{CirLoeWei18} for the extension to the case of the interior of $C$ being empty. 

\begin{theorem}\label{th3}
Let problem \eqref{P} with $g$ being polyhedral have an optimal solution and let $h$ be closed. Then, for the associated problem \eqref{ConcMinDC}, assumptions (M) and (B) are satisfied for $P = \epi g$ and the polyhedral convex cone $C = 0^+\epi g$.
\end{theorem}
\begin{proof}
The set $C=0^+\epi g$ is a polyhedral convex cone. Obviously, (B) holds. It remains to show (M). Let $(x,r),(y,s) \in \R^n\times \R$ such that
$$\begin{pmatrix}
	y-x\\s-r
\end{pmatrix} \in C = 0^+\epi g \overset{\text{Cor.\ \ref{cor1} (iii')}}{\subseteq} 0^+\epi h.$$ 
If $x \notin \dom g$, then $f(x,r) = -\infty \leq f(y,s)$. 
We have
$$
	\underbrace{(x,h(x))}_{\in \epi h} + \underbrace{(y-x,s-r)}_{\in 0^+\epi h} = (y,h(x)+s-r) \in \epi h.
$$
By definition of $\epi h$, we obtain $h(x)+s-r \geq h(y)$ and hence
$r-h(x) \leq s-h(y)$, which proves (M).
\end{proof}

\begin{remark}
	In the previous theorem, the assumption of $h$ being closed can be omitted if the definition of $f$ in \eqref{obj} is replaced by
	\begin{equation}\label{obj1}
		f(x,r)\colonequals \begin{cases} r-h(x) & \text{ if } (x,r)\in \dom g \\ -\infty & \text{ otherwise} \end{cases}.
	\end{equation}
	The proof is similar by using Theorem \ref{th1} (iii) instead of Corollary \ref{cor1} (iii').
\end{remark}


The cone $C$ in the previous theorem is not necessarily pointed, as required for the solution methods of \cite{CirLoeWei18}, see above. However, pointedness can be achieved by a reformulation of problem \eqref{P}: Denote by $L$ the {\em lineality space} of the convex function $g$ which is defined by
$$ L \colonequals \left\{ x \in \R^n \st \exists r \in \R: \; \begin{pmatrix}
x \\ r	\end{pmatrix} \in \lineal (\epi g) \right\}.$$
Let $L^\bot$ be the orthogonal complement of $L$. For some fixed $\bar x \in \dom g$ we define
\begin{equation}\label{eq8}
	\bar g (x) \colonequals \begin{cases} g(x) & \text{ if } x-\bar x \in L^\bot\\ \infty & \text{ otherwise} \end{cases}.
\end{equation}
We denote by ($\rm \overline{DC}$) the polyhedral d.c.\ optimization problem \eqref{P} where $g$ is replaced by $\bar g$.

\begin{proposition}\label{p6} Let \eqref{P} with $g$ being polyhedral have an optimal solution. Then $\rm(\overline{DC})$ has an optimal solution and every optimal solution of $\rm(\overline{DC})$ is also an optimal solution of \eqref{P}.
\end{proposition}
\begin{proof}
	We have $\dom \bar g \neq \emptyset$, $\dom \bar g \subseteq \dom g$ and $0^+ \epi \bar g \subseteq 0^+ \epi g$. Theorem \ref{th1} yields the first statement. Now let $x^0$ be an optimal solution of the modified problem  $\rm(\overline{DC})$. The point $x^0$ is feasible for \eqref{P}. Assume there is some $\tilde x \in \dom g$ such that $g(\tilde x)-h(\tilde x) < g(x^0)-h(x^0) = \bar g(x^0)-h(x^0)$. Define 
	$$\hat x \colonequals  (\{\tilde x\} + L)\cap(\{\bar x\}+L^\bot) \in \dom \bar g.$$
We show that $g(\tilde x)-h( \tilde x) = \bar g(\hat x)- h(\hat x)$. Indeed, we have $\hat x - \tilde x \in L$, hence there is some $r \in \R$ such that 
$$ \begin{pmatrix}
	\hat x-\tilde x\\ r
\end{pmatrix} \in \lineal(\epi g).$$
From Theorem \ref{th1} (iii) we conclude that
$$ \begin{pmatrix}
	\hat x-\tilde x\\ r
\end{pmatrix} \in \lineal(\epi h \cap (\dom g \times \R))= \lineal(\epi (h \vert_{\dom g})),$$
where $h\vert_{\dom g}$ is the function that coincides with $h$ on $\dom g$ and is $\infty$ elsewhere. From \cite[Theorem 8.8]{Rockafellar72} we conclude that
\begin{equation}\label{eq_13}
	g(x + \lambda (\hat x - \tilde x)) = g(x) + \lambda r
\end{equation} 
for all $x \in \R^n$ and all $\lambda \in \R$.
Likewise we get
\begin{equation}\label{eq_14}
	h(\underbrace{x + \lambda (\hat x - \tilde x)}_{\in \dom g}) = h(x) + \lambda r
\end{equation}
for all $x \in \dom g$ and all $\lambda \in \R$. 
We obtain
$$ g(\tilde x) - h(\tilde x) \stackrel{\eqref{eq_13},\eqref{eq_14}}{=} g(\tilde x + \lambda (\hat x - \tilde x)) - h(\tilde x + \lambda (\hat x - \tilde x)) \stackrel{\lambda = 1}{=} g(\hat x) - h(\hat x).$$
Since $\hat x \in \dom \bar g$, we have $g(\hat x)=\bar g(\hat x)$. Together we have $\bar g(\hat x)- h(\hat x) < \bar g(x^0)-h(x^0)$ which contradicts the assumption that $x^0$ is optimal for $\rm(\overline{DC})$.
\end{proof}

The following example shows that an optimal solution of \eqref{P} is not necessarily an optimal solution of $\rm(\overline{DC})$.

\begin{example} Let $g,h:\R\to \R$, $g\equiv 0$ and $h\equiv 0$. Then $L=\R$, $L^\bot=\{0\}$. We have $\bar g(0)=0$ and $\bar g(x)=\infty$ if $x\neq 0$. Thus $0$ is the only optimal solution of $\rm(\overline{DC})$ but every $x\in \R$ is optimal solution of \eqref{P}.
\end{example}

Summarizing the results, we solve \eqref{P} with $g$ being polyhedral by the following procedure:
\begin{itemize}
	\item[(1)] Check whether \eqref{P} has an optimal solution or not using Theorem \ref{th1}, if not, stop.
	\item[(2)] Determine $\bar x \in \dom g$  and $L^\bot$ in order to define the function $\bar g$ in \eqref{eq8}.
	\item[(3)] Solve \eqref{ConcMinDC} with $g$ replaced by $\bar g$ using the methods of \cite{CirLoeWei18}.
\end{itemize}

In the case where $h$ is polyhedral, we need to assume additionally that $g$ is closed and \eqref{p3} has an optimal solution for every $y \in \dom g^*$. Then we can check whether or not \eqref{P} has an optimal solution by using Theorem \ref{th2}. If so, by Theorem \ref{th1} we know that \eqref{D} has an optimal solution. An optimal solution $y$ of \eqref{D} can be obtained by the same method (steps (2) and (3) of the above procedure) but applied to \eqref{D} rather than \eqref{P} (replace $g$ by $h^*$ and $h$ by $g^*$). Finally, we solve \eqref{p3}, which provides an optimal solution of \eqref{P}.

If both $g$ and $h$ are polyhedral, \eqref{P} can be solved by two different methods: We speak about the {\em primal method} in case we use the method where $g$ is required to be polyhedral. The term {\em dual method} refers to the method where $h$ is required to be polyhedral. Furthermore there are two different tests for existence of an optimal solution of \eqref{P}. The test in Corollary \ref{cor48} (c) is referred to as {\em primal existence test} whereas (d) in Corollary \ref{cor48} is called {\em dual existence test}.

\section{Numerical results}

We implemented the results of this article in Matlab 9.6 by using {\it bensolve tools}, version 2.3, see \cite{bt_paper,bt}. The code and the test instances are available at \url{http://tools.bensolve.org/dcsolve}. By two (new) commands {\tt dcsolve} and {\tt dcdsolve} the user can run, respectively, the primal and dual method described in the previous section. The input arguments of both commands are two arbitrary polyhedral convex functions $g$ and $h$ in the usual format of Bensolve Tools. Both commands solve arbitrary polyhedral d.c.\ optimization problems (of small size) or certify that no optimal solution exists. 

The following two numerical examples were run on a computer with Intel\textregistered\ Core\texttrademark\ i5
CPU with 3.3 GHz.

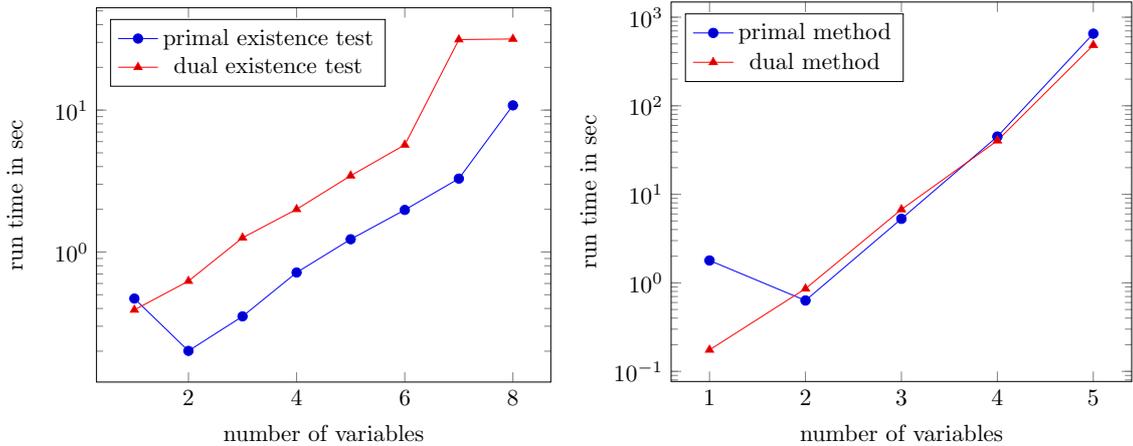
\begin{figure}[h]
	\resizebox{0.5\textwidth}{!}{%
	\begin{tikzpicture}[scale=0.5]
		\begin{axis}[
			ymode=log,
			xlabel=number of variables,
			ylabel=run time in sec,
			legend pos=north west]
			\addplot+[mark=*] table {ex1_prim_exist_test.dat};
			\addlegendentry{primal existence test}
			\addplot+[mark=triangle*] table {ex1_dual_exist_test.dat};
			\addlegendentry{dual existence test}
		\end{axis}
	\end{tikzpicture}
	}
	\resizebox{0.5\textwidth}{!}{%
	\begin{tikzpicture}
		\begin{axis}[
			ymode=log,
			xlabel=number of variables,
			ylabel=run time in sec,
			legend pos=north west]
			\addplot+[mark=*] table {ex1_prim_sol.dat};
			\addlegendentry{primal method}
			\addplot+[mark=triangle*] table {ex1_dual_sol.dat};
			\addlegendentry{dual method}
		\end{axis}
	\end{tikzpicture}
	}
	\caption{Numerical results for Example \ref{ex1}. The number of columns of $G$ and $H$ are $m_G=20$ and $m_H=15$.
	Left: Run time of test for existence of solutions in dependance of the number of variables.
	Right: Run time of primal and dual solution algorithms (without existence test) in dependance of number of variables.}
	\label{fig_ex1}
\end{figure}

\begin{example}\label{ex1}
	Let $A \in \R^{n \times m}$ be a matrix and denote by $a^i$ its columns. We define a polyhedral convex function $f_A:\R^n \to \R$ by
	$$ f_A(x) \colonequals \sum_{i=1}^m \|x-a^i\|_1,$$
	where $\|y\|_1 = \sum_{j=1}^n |y_j|$ denotes the sum norm of a vector $y$. Given two matrices $G \in \R^{n\times m_G}$ and $H \in \R^{n \times m_H}$ we consider the polyhedral d.c.\ optimization problem 
	\begin{equation}\label{dc_ex1}
		\min_{x\in \R^n} f_G(x)-f_H(x).
	\end{equation}
	Problems of this type occur in locational analysis, see e.g. \cite{LoeWag17} and the references therein.
	In Figure 1 numerical results are depicted for matrices $G$ and $H$ with components $g_{ij}=sin(i+j)$ and $h_{ij}=cos(i+j)$ (just to make the results easily reproducible in comparison to random numbers). The recession cone of $\epi f_A$ is just the recession cone of $\epi (m \|\cdot\|_1)$ and $\dom f_A = \R^n$.
	Thus, by Corollary \ref{cor1}, a solution of \eqref{dc_ex1} exists if and only if $m_G \geq m_H$.
	Figure \ref{fig_ex1} (left) shows the run time of a numeric verification of this fact for some problem instances by checking the conditions of Corollary \ref{cor1} (primal existence test) and Theorem \ref{th2} (dual existence test).
	Figure \ref{fig_ex1} (right) shows the run time of the primal and dual solution methods.

\end{example}

The following example from \cite{Ferrer2014} was solved in \cite{LoeWag17} using the MOLP solver Bensolve \cite{bensolve_paper}.
We implemented the algorithms of \cite{LoeWag17} with Bensolve Tools and compare them with our algorithms. 
While the methods of \cite{LoeWag17} compute all vertices of $\epi g$ in the primal algorithm and all vertices of $\epi h^*$ in the dual algorithm, we compute only part of these vertices by using the {\tt qcsolve} command of Bensolve Tools.
We observe a better performance for the bigger instances.

\begin{figure}[h]
	\resizebox{0.5\textwidth}{!}{%
	\begin{tikzpicture}
		\begin{axis}[
			ymode=log,
			xlabel=number of variables,
			ylabel=time in sec,
			legend pos=north west]
			\addplot+[mark=*] table {ex2_prim_exist_test.dat};
			\addlegendentry{primal existence test};
			\addplot+[mark=triangle*] table {ex2_dual_exist_test.dat};
			\addlegendentry{dual existence test}
		\end{axis}
	\end{tikzpicture}
	}
	\resizebox{0.5\textwidth}{!}{%
	\begin{tikzpicture}
		\begin{axis}[
			ymode=log,
			xlabel=number of variables,
			ylabel=time in sec,
			legend pos=north west]
			\addplot+[mark=*, mark options={draw=blue, fill=blue}] table {ex2_prim_sol.dat};
			\addlegendentry{primal method};
			\addplot+[mark=triangle*, mark options={draw=red, fill=red}] table {ex2_dual_sol.dat};
			\addlegendentry{dual method}
			\addplot+[blue, mark=*, mark options={draw=blue, fill=blue, solid}, dashed] table {ex2_prim_alt.dat};
			\addlegendentry{primal alg.\ of \cite{LoeWag17}}
			\addplot+[red, mark=triangle*, mark options={draw=red, fill=red, solid}, dashed] table {ex2_dual_alt.dat};
			\addlegendentry{dual alg.\ of \cite{LoeWag17}}
		\end{axis}
	\end{tikzpicture}
	}
	\caption{Numerical results for Example \ref{ex2}. Left: existence tests.
	Right: Comparison of our algorithms (without existence test)
	with the methods from \cite{LoeWag17} (also without existence test). 
	We observe that both dual methods perform better than the primal ones. 
	This is due to the easier structure of $h$ in comparison to $g$.
	We observe that our methods perform better in case of instances
	with sufficiently many variables.
	}
	\label{fig_ex2}
\end{figure}
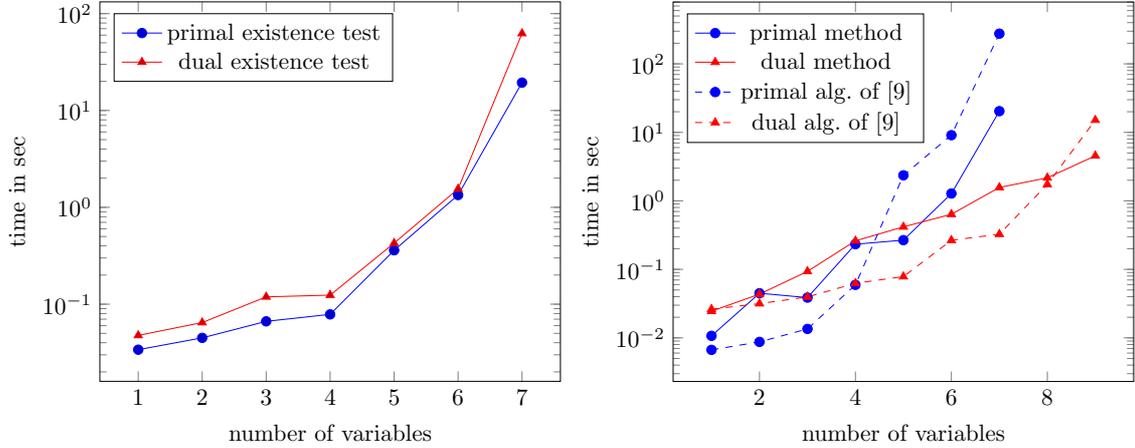

\begin{example} \label{ex2}
	Consider the polyhedral d.c.\ optimization problem \eqref{P} with
	$$ g(x) = |x_1-1|+200 \sum_{i=2}^n  \max \left\{0,|x_{i-1}|-x_i\right\}$$
	and
	$$ h(x) = 100 \sum_{i=2}^n  \left(|x_{i-1}|-x_i\right).$$
	One can easily verify that the all-one vector provides an optimal solution. 
	In Figure \ref{fig_ex2}, the run time of the primal and dual solution method
	proposed in this article is compared to the primal and dual method of \cite{LoeWag17}.
\end{example}

\section{Summary}

We characterized the existence of global optimal solutions of polyhedral d.c.\ optimization problems in Theorem \ref{th1}, Theorem \ref{th2} and Corollary \ref{cor48} depending on whether the first, the second, or both components of the objective function are polyhedral. We provided a solution procedure based on both an existence tests and a reformulation of the polyhedral d.c.\ optimization problem into a quasi-concave minimization problem. Numerical experiments were run for the case where both components of the objective function are polyhedral.

\end{document}